\documentclass[reqno,11pt]{amsart}
\usepackage{amsmath, latexsym, amsfonts, stmaryrd, mathabx, amsthm, amscd, array, caption}
\usepackage[toc,page]{appendix}
\usepackage[utf8]{inputenc}
\usepackage{graphics,epsf,psfrag,epsfig}
\usepackage[colorlinks=true, pdfstartview=FitV, linkcolor=blue, citecolor=blue, urlcolor=blue,pagebackref=false]{hyperref}
\usepackage{bm}
\usepackage{a4wide}
\usepackage[english]{babel}
\usepackage{amssymb}
\usepackage{amsfonts, dsfont}
\usepackage{amsthm}
\hypersetup{colorlinks,citecolor=red,filecolor=purple,linkcolor=blue,urlcolor=black}
\usepackage[T1]{fontenc}
\usepackage{enumerate}
\usepackage{relsize}
\usepackage{mathtools}
\usepackage{bbm}

\newcommand{\N}{\mathbb N}
\newcommand{\R}{\mathbb R}
\newcommand{\Z}{\mathbb Z}
\newcommand{\E}{\mathbb E}
\newcommand{\cE}{\mathcal E}
\newcommand{\CC}{{\rm CC}}

\newcommand{\CCc}{\widehat{\rm CC}}
\newcommand{\CCav}{{\rm CC_{av}}}

\renewcommand{\P}{\mathbb P}

\newcommand{\1}[1]{\mathds 1_{\left\{#1\right\}}}

\newcommand{\e}{{\rm e}}
\renewcommand{\l}{{\nu}}
\renewcommand{\a}{\alpha}

\newcommand{\pg}{\mathbb P}
\newcommand{\pp}{\mathrm P}
\newcommand{\ep}{\mathrm E}
\newcommand{\vp}{\mathrm V}
\newcommand{\pw}{P}
\newcommand{\ew}{E}
\newcommand{\pe}{ P^\eta}

\def\be{\begin{eqnarray}}
\def\ee{\end{eqnarray}}
\def\ben{\begin{eqnarray*}}
\def\een{\end{eqnarray*}}

\newtheorem{proposition}{Proposition}[section]

\newtheorem{lemma}[proposition]{Lemma}

\newtheorem{theorem}[proposition]{Theorem}
\newtheorem{remark}[proposition]{Remark}

\newtheorem*{assumptionA2'}{Assumption A2'}

\newcounter{example}
\title[Continuum scale-free percolation: degree and clustering coefficient]{Scale-free percolation in continuum space: \\ quenched degree and clustering coefficient}

\author[J.\ Dalmau]{Joseba Dalmau}
\address{Joseba Dalmau.  {Institute of Mathematical Sciences, NYU Shanghai.
		Geography Building, 3663 North Zhongshan Road, Shanghai, China}}
\email{jdc16@nyu.edu}

\author[M.\ Salvi]{Michele Salvi}
\address{Michele Salvi.  {CMAP, D\'epartement de Math\'ematiques Appliqu\'ees, \'Ecole Polytechnique, Route de Saclay, 91128 Palaiseau Cedex, France
		and 
		MaIAGE, INRA, Université Paris-Saclay, 78350 Jouy-en-Josas, France
		}}
\email{michele.salvi@polytechnique.edu}

\begin{document}

\begin{abstract}
	Spatial random graphs capture several important properties of real-world networks.
	We prove quenched results for the continuum space version of scale-free percolation introduced in \cite{DW18}. This is an undirected inhomogeneous random graph whose vertices are given by a Poisson point process in $\R^d$. Each vertex is equipped with a random weight and the probability that two vertices are connected by an edge depends on their weights and on their distance. Under suitable conditions on the parameters of the model, we show that, for almost all realizations of the point process, the degree distributions of all the nodes of the graph follow a power law with the same tail at infinity. We also show that the averaged clustering coefficient of the graph is {\it self-averaging}. In particular,  it is almost surely equal to the the annealed  clustering coefficient of one point, which is a strictly positive quantity.
	

	\medskip
	
	\noindent  \emph{AMS  subject classification (2010 MSC)}: 
	05C80 
	05C63 
	05C82  	
	05C90  	
	
	\smallskip
	

\end{abstract}

\maketitle

\section{Introduction}
Random graphs are a powerful tool to model real-world large networks such as the internet \cite{AJB99}, 
telecommunication networks \cite{HABDF09}, social networks \cite{NWS02}, neural networks \cite{LHCS00}, transportation networks \cite{KKGB10}, financial systems \cite{CMS10} and many more (see e.g.~\cite{N03, V16} for overviews). The idea is to overcome the intractability of a given network, for example because of its size, by mimicking its most interesting features with a probabilistic model.  In particular, three main characteristics have been highlighted in the recent  literature that are  often observed in real-life  (see e.g.~\cite{CF16, HHJ17, JM15}):
\begin{itemize}
	\item {\it Scale-free}: A graph is said to exhibit the scale-free property if the degree (i.e.~the number of neighbors) of its vertices follows a power law.  This results in the presence of hubs, that is, nodes with a very high degree.
	
	\smallskip
	
	\item {\it Small-world}: When sampling two vertices at random, their graph distance (the minimum number of edges to cross to go from one vertex to the other) is typically of logarithmic order on the total number of  nodes. If the nodes of the graph are embedded in a metric space, for example in $\R^d$, this translates into saying that two nodes at Euclidean distance $D$ are with high probability at graph-distance $O(\log D)$. 
	
	\smallskip
	
	\item {\it High clustering coefficient}: When two vertices share a common neighbor, there is a ``high'' probability that they are linked, too.
\end{itemize}
These properties are of great importance when studying, for example, the spread of information or infections (see e.g.~\cite{PV01, MN00, CL14, REIMT15}) or other related random processes on the network (like first passage percolation \cite{KL18}, Ising model \cite{BW00} or random walks \cite{HHJ17}).
Unfortunately, classical models of random graphs do not exhibit these three properties at once. Just to give a few examples, the Erd\H{o}s-R\'enyi model has only the small-world property, the Chung-Lu \cite{CL02}, Norros-Reittu \cite{NR06} and preferential attachment \cite{BA99} models are scale-free and small-world, but all have vanishing clustering coefficient as the number of nodes tends to infinity.
As a rule of thumb, one can think that regular, nearest-neighbor lattices have high clustering but long distances, while classic random graphs have low clustering and small distances. The Watts-Strogatz network \cite{WS98} was one of the first attempts to ``artificially'' build a random graph that boasts high clustering and small-world, but it lacks the scale-free property.

\smallskip

Spatial random graphs, i.e.~graphs whose vertices are embedded in some metric space, offer a possible solution to the shortcomings of classical models. 
 While in the Physics literature they have been quite extensively studied (see \cite{B11} for a broad review on the topic), very few models have been analyzed in a fully rigorous, mathematical fashion. Among those which exhibit or are conjectured to exhibit the three properties mentioned above, we mention the hyperbolic random graph 
 \cite{GPP12}, the spatial preferential attachment model \cite{JM15},
the geometric inhomogeneous random graph \cite{BKL18} 
and scale-free percolation.


\smallskip

Scale-free percolation has been introduced in \cite{DVH13} and can be considered a combination of long range percolation (see e.g.~\cite{B02, B04})  and inhomogeneous random graphs as the Norros-Reittu model.
The vertices of the graph lie on the $\Z^d$ lattice and random weights $\{W_x\}_{x\in\Z^d}$ are assigned independently to each of them. The distribution of the weights follows a power law: we have $P(W>w)=w^{-(\tau-1)}L(w)$ for some $\tau>1$ and $L$ function slowly varying at infinity. Finally, any two vertices $x,y\in\Z^d$ are linked by an unoriented 
edge with probability 
\begin{align*}
p_{x,y}=1-\e^{-\lambda W_xW_y/\|x-y\|^\alpha}\,,
\end{align*}
where $\lambda, \alpha>0$ are two parameters of the model and $\|\cdot\|$ is the Euclidean distance. In \cite{DHW15} it is argued that scale-free percolation is a suitable model, for example, for  the interbank network presented in \cite{SBAGB07}.
 In  \cite{DW18} the authors have introduced the continuum space counterpart of scale-free percolation, where the vertices of the graph are sampled according to a homogeneous Poisson point process of intensity $\l>0$ in $\R^d$. On the one hand, this additional source of randomness makes the model more flexible and possibly more suitable for applications to real-world networks. On the other hand, technical difficulties emerge, especially when dealing with fixed configurations of the point process. The probability of linking two vertices is the same as in \cite{DVH13}, but the authors prefer to restrict to a Pareto distribution for the weight distribution: $P(W>w)= w^{-(\tau-1)}\1{w>1}$ (note that in the present paper we chose to stick to the notation of  \cite{DVH13}). We point out that the results appearing in  \cite{DW18} are {\it annealed} -- that is, obtained after integration against the underlying Poisson point process. 

\smallskip

Both original scale-free percolation model of \cite{DVH13} and the continuum scale-free percolation of \cite{DW18} share the following features. If $\alpha\leq d$ or $\gamma:=\alpha(\tau-1)/d\leq 1$, then the nodes of the graph have almost surely infinite degree. If instead $\min\{\alpha,(\tau-1)\alpha\}>d$, then the degree of the vertices follows a power law of index $\gamma$ (hence the graph has the {\it scale-free} property). 
Let  now $\lambda_c$ be the percolation threshold of the graph, that is, the value such that for  $\lambda<\lambda_c$ all the connected components of the graph are finite and for $\lambda>\lambda_c$ there exists almost surely a (unique) infinite connected component. Assume $\min\{\alpha,(\tau-1)\alpha\}>d$. Then, again for both models, the following holds: for $d\geq 2$, if $\gamma\in(1,2)$ then $\lambda_c=0$, while if $\gamma>2$, then $\lambda_c\in(0,\infty)$; for $d=1$, if $\gamma\in(1,2)$ then $\lambda_c=0$, if $\gamma>2$ and $\alpha\in(1,2]$ then $\lambda_c\in(0,\infty)$ and if $\min\{\alpha,(\tau-1)\alpha\}>2$ then $\lambda_c=\infty$. Once the percolation properties have been established, 
one can talk about graph-distances (for the discrete-space model the results on percolation and distances of  \cite{DVH13} have been complemented in \cite{DHW15}). The graph-distance between two nodes $x$ and $y$ belonging to the same connected component of the graph is the minimum number of edges that one has to cross to go from $x$ to $y$. 
 Assume again $\min\{\alpha,(\tau-1)\alpha\}>d$. Again for both models,
when the vertex degrees have infinite-variance (corresponding to $\gamma\in(1,2)$), the graph-distance between two points belonging to the infinite component of the graph grows like the $\log\log$ of their Euclidean distance (in this case the graph is said to exhibit the  {\it ultra-small world} property). When the vertex degrees have finite variance ($\gamma >2$) and $\lambda>\lambda_c$, two cases are possible: if $\alpha\in(d,2d)$, then the graph distance of two points at Euclidean distance $D$ grows roughly as $(\log D)^\Delta$ for some constant $\Delta>0$ still not precisely known ({\it small-world} property);   if $\alpha>2d$ the graph distance is bounded from below by a constant times the Euclidean distance. We point out that the regime where the degrees have infinite ($\gamma\in(1,2)$ in our case) is usually the relevant one for applications (see, e.g., \cite{AJB99, JTAOB00}).

%

\smallskip

%

%
%
%
%

\subsection{Our contribution}
In the present paper we aim at proving {\it quenched} results for scale-free percolation in continuum space, that is, statements that hold for almost every realisation $\eta$ of the underlying Poisson point process. This is one of the main differences with  \cite{DW18}: taking the annealed measure therein allows the authors to calculate relevant quantities explicitly. Unfortunately, in most cases, annealed properties give little information about  a {\it given} configuration $\eta$. Besides their mathematical interest, quenched results are also important because they guarantee, for example, that a simulation of the graph will always exhibit a given feature. We also assume more general distributions for the weights
, not restricting to Pareto.

\smallskip

We start by analysing the tail of the degree distribution. The regime of parameters that ensures infinite degree for all the vertices is the same as in \cite{DVH13} and \cite{DW18}, see Theorem \ref{infinitedegree}. More interesting is the case of finite degrees.
In  \cite{DW18}, under the condition of weights that follow a Pareto distribution, it was possible to calculate the whole distribution of the annealed degree of a  vertex just by integrating over the Poisson measure. While there is clearly no chance that this works for all the vertices of the graph for a fixed configuration $\eta$, we show that the behavior of the tail of the degree of {\it all} vertices is the same (hence the graph is {\it scale-free}).  More precisely, we show (see Theorem \ref{degree}) that for almost every $\eta$ the following holds: for all $x\in\eta$ there exists a slowly varying function $\ell(\cdot)=\ell(\cdot,\eta,x)$ such that the probability that $D_x$, the degree of $x$, is bigger than some $s>0$ is equal to $\ell(s)s^{-\gamma}$, where $\gamma :=\alpha(\tau-1)/d$.
Our proof follows the strategy of \cite{DVH13} rather than \cite{DW18}, but a major difference emerges already when calculating the expectation of the degree of a vertex given the value $w>0$ of its weight. While in \cite{DVH13} it is shown by direct computation that the expectation is equal up to a constant to $\xi w^{d/\alpha}$ for an explicit $\xi>0$, in our case there is  a correction of the order $w^{d/(2\alpha)}$ that accounts for the fluctuations of the  Poisson point process  (see Proposition \ref{marmotta}). We apply concentration inequalities in combination with the so-called Campbell theorem to achieve this result. 
With some further effort we prove that the result holds for all the points of $\eta$ at once.

\smallskip

We move then to the clustering coefficient of the graph, which has not been studied neither in \cite{DVH13} nor in \cite{DW18}. For a given graph, the {\it local} clustering coefficient $\CC(x)$ of a node $x$ is given by $\Delta_x$, the number of triangles with a vertex in $x$ (that is, triplets of edges of the form $(x,y),(y,z),(z,x)$), divided by $D_x(D_x-1)/2$. This second quantity represens the number of ``possible'' triangles with a vertex in $x$, also called open triangles. The  {\it averaged} clustering coefficient of a finite graph is the average of $\CC(x)$ over all its vertices (there is also a notion of {\it global} clustering coefficient that we do not analyze here). For an infinite graph like ours, we define the averaged clustering coefficient as the limit (if it exists) for $n\to\infty$  of $\CC_n$, where $\CC_n$ is the average of the local clustering coefficients of the vertices inside the $d$-dimensional box of side-length $n$ centered at the origin.  We show that $\CC_n$ is {\it self-averaging} (a similar property has been recently proved  in \cite{SVV18} for the hyperbolic random graph when the number of nodes goes to infinity), so that its limit exists and is almost surely equal to the expectation of the local clustering coefficient of the origin $0$ obtained under the Palm measure associated with the Poisson point process with a point added at $0$ (see Theorem \ref{thmcc}). We also show that this quantity is strictly positive ({\it high clustering coefficient}).
The proof consists in dividing the box of size $n$ into mesoscopic boxes of side-length $m<n$. We then approximate the clustering coefficient of each $m$-box by a truncated clustering coefficient that does not take into account neither vertices that are close to the border of the boxes nor edges connecting nodes in different $m$-boxes. In doing so, we obtain independent truncated clustering coefficients in each mesoscopic box and we can use the law of large numbers as $n\to\infty$. In decorrelating the clustering coefficient of each $m$-box, we must control the correlation of the local clustering coefficient of vertices lying in different $m$-boxes. To do so, we use a second moment approach in combination with the Slivnyak-Mecke theorem. 
The technique we use should be applicable also to the original scale-free percolation model of \cite{DVH13}. It is interesting to note that the positivity of the clustering coefficient does not depend on the local connectivity properties of the graph, see Remark \ref{marco}, nor on $\l$.

\smallskip

Finally we point out that, at least under the hypothesis of Pareto weights, the results on percolation and graph-distances presented in the introduction have annealed probability equal to $1$, see \cite[Theorem 3.2, Theorem 3.6]{DW18}. Hence, they hold also in the quenched sense (see Remark \ref{rimarco} to fill a little difference between our parameters and those of \cite{DW18}). It follows that under the right range of parameters, scale-free percolation in continuum space fulfils, for almost every realization of the Poisson point process $\eta$, all the three properties presented in the introduction at once,  almost surely. This suggests  that the model is a good candidate for modelling real spatial networks.  We believe, for example, that it encloses the main features of the cattle trading network in France (see \cite{DEV14}), which was the original motivation for our work. The nodes of the graph are given by farms or cattle markets in the country, so that the hypothesis of vertices placed according to a Poisson point process reflects in a more realistic way than $\Z^d$ the geographic irregularities. An edge is placed between two holdings if there has been a transaction between the two during, say, the year. A comprehension of the topology of this network might be of paramount importance when studying the possible outbreak  of an infection in the cattle population. This kind of analysis will be presented in a subsequent work. 

\smallskip

\subsection{Structure of the paper}
We will introduce the model and some notation in Section \ref{model}. We will present the main results on the degree distribution (Theorem \ref{infinitedegree} and Theorem \ref{degree}) and on the clustering coefficient (Theorem \ref{thmcc}) in Sections \ref{mainresults} and \ref{mainresults2}. We deal with the proof of the infinite degree case in Section \ref{secinfinitedegree} and of the finite degree case in Section \ref{secpolynomialdegree}. Section \ref{secclusteringcoefficient} is dedicated to the proof of the positivity of the clustering coefficient. In Appendix \ref{campbell} we recall Campbell and Slivnyak-Mecke theorems.

\bigskip

\section{Model and main results}
\subsection{The model}\label{model}
Scale-free percolation in continuum space is a random variable on the space of all simple spatial undirected random graphs with vertices in $\R^d$. 
To construct an instance of the graph $G=(V,E)$ we proceed in three steps:
\begin{itemize}
	\item We sample the nodes $V$ of the graph according to a homogeneous Poisson point process in $\R^d$ of intensity $\l>0$. We denote by $\pp$ its law and by $\ep$ the expectation w.r.t.~$\pp$. A configuration of the point process will be denoted by $\eta\in(\R^d)^\N$.
	\item We assign to each vertex $x\in\eta$ a random weight  $W_x>0$. The weights are i.i.d.~with law $\pw$. The distribution function $F$ associated to $\pw$ is regularly varying at infinity with exponent $\tau-1$, that is
	\begin{align}\label{codapesi}
	1-F(w)=\pw(W>w)=w^{-(\tau-1)}L(w)
	\end{align}
	where $ L(\cdot)$ is a function that is slowly varying at $+\infty$. 
	\item We finally draw the edges $E$ of the random graph via a percolation process. For every pair of points $x,y\in\eta$ with weights $w_x,w_y$, the undirected edge $(x,y)$ is present with probability
	\begin{align}\label{connectingprob}
	1-\exp\frac{w_x w_y}{\|x-y\|^{\alpha}}\,,
	\end{align}
	where $\|\cdot\|$ denotes the Euclidean norm in $\R^d$ and $\alpha>0$ is a parameter. When two vertices $x,y$ are connected by an edge we write $x\leftrightarrow y$ (if they are not connected we write $x\not\leftrightarrow y$). For the event that a vertex $x$ is connected to some other point in a region $A\subseteq\R^d$ we write  $x\leftrightarrow A$. 
\end{itemize}

\begin{remark}\label{rimarco}
	In \cite{DW18} an additional parameter $\lambda>0$ appears in the definition of the linking probabilities: $p_{x,y}=1-\exp\{-\lambda W_xW_y\|x-y\|^{-\alpha}\}$. In this paper we prefer to fix $\lambda=1$ without loss of generality. Indeed, having parameters $\lambda=\tilde\lambda$ and $\l=\tilde\l$ in the model presented in \cite{DW18} is completely equivalent to choosing $\l=\tilde\l \tilde\lambda^{d/\alpha}$ in our case.
\end{remark}

For a given realization $\eta$ of the Poisson point process we will often perform the last two steps at once under the {\it quenched} law $\pe$, with $E^\eta$ denoting the associated expectation. This is the law of the graph once we have fixed  $\eta$. We write $\pg$ for the {\it annealed} law of the graph, that is, $\pg=\pp\times\pe$, and $\E$ for the corresponding expectation.
For $x\in\R^d$, we indicate by $\pp_x$   the law of the Poisson point process conditioned on having a point in $x$; analogously, for $x,y\in\R^d$, the measure $\pp_{x,y}$ is obtained by conditioning on having one point in $x$ and one in $y$. We will not enter in the details of ($n$-fold) Palm measures, but we point out that in both cases the conditioning does not influence the rest of the Poisson point process, see e.g.~\cite[Chap.~13]{DVJ}. Consequently we will call $\P_x =\pp_x\times P^\eta=\pp\times P^{\eta\cup \{x\}}$ the law of the graph conditioned on having a point at $x\in\R^d$ and $\E_x$ the associated expectation. Analogously we will write $\P_{x,y}=\pp_{x,y}\times P^\eta=\pp\times P^{\eta\cup \{x\}\cup\{y\}}$ and $\E_{x,y}$ for the expectation with respect to $\P_{x,y}$.

\smallskip

We denote by $0$ the origin of $\R^d$. We call $B_n$ the $d$-dimensional box of side-length $n>0$ centered in $0$. We let $V_n=V_n(\eta)$ be the  set of points of $\eta$ that are in $B_n$ and $N_n=N_n(\eta)$ their number. Finally, $\mathcal B_r(x)$ denotes the Euclidean ball of radius $r>0$ centered in  $x\in\R^d$.


\subsection{Results on the degree}\label{mainresults}
For each $x\in\eta$ we call $D_x:=\#\{y\in\eta:\,y\leftrightarrow x\}$ the degree of a vertex $x$ in the random graph, 
that is, the random variable counting the number of neighbors of $x$.
We start by showing that for certain values of the parameters 
the graph is almost surely degenerate, in the sense that all of its vertices have infinite degree.

\begin{theorem}\label{infinitedegree}
	Suppose one of the following conditions is satisfied:
	\begin{itemize}
		\item[(a)]  $\alpha\leq d$ ;
		\item[(b)] the weight distribution satisfies 
		\begin{align}\label{rocco}
		1-F(w)\geq c\,w^{-(\tau-1)}
		\end{align}
		for some $c>0$ and $\tau>1$ such that $\gamma:=\alpha(\tau-1)/d\leq 1$ .
	\end{itemize}
	Then for $\pp$--almost every $\eta$,  $P^\eta$--almost surely
	all points in $\eta$ have infinite degree.
\end{theorem}
We point out that the proof of Theorem \ref{infinitedegree} follows immediately from the analogous result of \cite{DW18} when the weights have  a Pareto distribution. We will provide an alternative proof that, besides covering more general distribution functions $F$, serves as a warm-up for the kind of techniques we will use more intensively later on, based on Campbell's theorem.

\smallskip

We move to the analysis of the regime where the degree of the nodes is almost surely finite. In this case, we show that for almost all $\eta$ the graph is scale-free and the degree of all points follows a power law with the same exponent.

\begin{theorem}\label{degree}
	Suppose that $\a>d$ and $\gamma:=\a(\tau-1)/d>1$. Then, for $\pp$--almost every $\eta$, the following holds: for each $x\in\eta$ there exists a slowly varying function $\ell(\cdot)=\ell(\cdot,\eta,x)$ such that
	\begin{align}\label{notorious}
	P^\eta(D_x>s)=s^{-\gamma}\ell(s)\,.
	\end{align}
\end{theorem}
\begin{remark}
The slowly varying function in \eqref{notorious} has clearly to depend on $x$ and $\eta$. However, it is not clear whether all $\ell$'s have the same asymptotic behavior. That is, with our proof it is not possible to say whether for $\pp$--almost every $\eta$ and $\eta'$, and for all $x\in\eta$ and $x'\in\eta'$, one has $\ell(s,\eta,x)/\ell(s,\eta',x')\to 1$ as $s\to\infty$.  
\end{remark}

\subsection{Results on the clustering coefficient.} \label{mainresults2}
The average clustering coefficient is usually defined for finite graphs as the average of the local clustering coefficients. More precisely, for a finite undirected graph $\mathcal G=(\mathcal V,\mathcal E)$ and a node $x\in \mathcal V$, we define the local clustering coefficient at $x$ as
\begin{align*}
\CC(x):=\frac{2\,\Delta_x}{D_x(D_x-1)}\,,
\end{align*}
where $\Delta_x$ is the number of closed triangles in the graph that have $x$ as one of their vertices (a closed triangle is a triplet of 
edges in $\mathcal E$ of the type $\{(y,z), (z,w), (w,y)\}$), with the convention that $\CC(x)$ is $0$ if $D_x$ is equal to $0$ or $1$. The quantity $D_x(D_x-1)/2$ can be interpreted as the number of open triangles in $x$ (an open triangle in a vertex $z$ is a couple of edges of the kind $\{(y,z),(z,w)\}$, without any requirement for the presence of the third edge that would close the triangle), so that $\CC(x)$ takes values in $[0,1]$. 
The averaged clustering coefficient of the graph $\mathcal G$ is
\begin{align*}
\CCav(\mathcal G):=\frac{1}{|\mathcal V|}\sum_{x\in \mathcal V} \CC(x)\,.
\end{align*} 
Clearly also $\CCav(\mathcal G)$ takes values in $[0,1]$.
It is not completely obvious what the definition of  averaged clustering coefficient should be for an infinite graph. We say that  an infinite spatial graph $\mathcal G=(\mathcal V,\mathcal E)$ embedded in $\R^d$ has averaged clustering coefficient $\CCav(\mathcal G)$ if the following limit exists
\begin{align*}
\lim_{n\to\infty} \frac{1}{| \mathcal V\cap B_n |}\sum_{x\in \mathcal V\cap B_n} \CC(x)=:\CCav(\mathcal G)\,,
\end{align*} 
(we work with integers for simplicity). It is possible to construct infinite graphs for which this limit does not exist. We also notice that this definition is in principle different from considering the limit of the clustering coefficients of the subgraphs of $\mathcal G$ obtained by considering only vertices in $B_n$, since $\CC(x)$ takes into account also the edges connecting points that lie out of $B_n$.

\smallskip

Going back to our model, we will call 
$$
\CC_n:=\frac{1}{N_n}\sum_{x\in V_n} \CC(x)\,,
$$
the random variable describing the averaged clustering coefficient inside the box of side $n$. 
Our main result of this section states that the averaged clustering coefficient of the scale-free percolation in continuum space exists,  is {\it self-averaging} and is strictly positive:
\begin{theorem}\label{thmcc}
	Suppose that $\a>d$ and $\gamma:=\a(\tau-1)/d>1$.
	For $\pp$--almost all configurations $\eta$, the average clustering coefficient exists $\pe$--almost surely and does not depend on $\eta$. It is given by
	\begin{align*}
	\CCav(G):=\lim_{n\to\infty}\CC_n\,
		=\E_0[\CC(0)]>0\,.
	\end{align*}
\end{theorem}
\begin{remark}\label{marco}
	Going through the proof of Theorem \ref{thmcc} one can get convinced that the positivity of the clustering coefficient does not depend on the local properties of the graph nor on $\l$, as one may think. For example, we can set to $0$ the probability of connecting two points that are at distance smaller than some $R>0$ and leave  \eqref{connectingprob} for points at distance larger than $R$. Then we would still have that the associated n-truncated clustering coefficient $\tilde\CC_n$ converges almost surely to $\E_0[\tilde\CC(0)]$, with $\tilde \CC(0)$ the corresponding local clustering coefficient. At the same time, $\E_0[\tilde \CC(0)]$ would still be strictly positive for any $R>0$ by similar arguments as in the proof of Lemma \ref{localcc}.
\end{remark}

\section{Degree}

Both in the infinite and in the finite degree case, for convenience we will first prove the properties of the degree of the vertices through an auxiliary random variable under $\pe$. We call it $D_0$ with slight abuse of notation, since $\pp$-almost surely the origin $0$ does not belong to $\eta$. It is convenient though to imagine to have a further point in the origin  and treat it like all the other points of the graph, so that when we talk about $D_0$ we can think $\pe$ to be substituted by $ P^{\eta\cup \{0\}}$. In a second moment we will transpose the properties of $D_0$ to the degree of the other vertices.

\subsection{Infinite degree}\label{secinfinitedegree} In this section we prove Theorem \ref{infinitedegree}. We start with the following proposition about $D_0$.

\begin{proposition}\label{obell}
	If condition (a) or (b)  in Theorem \ref{infinitedegree} are satisfied, then for $\pp$-a.a.~$\eta$ it holds
	$
	P^\eta(D_0=\infty)=1\,.
	$
\end{proposition}

\begin{proof}[Proof of Proposition \ref{obell}]
The statement will be proven by showing that, 
for each value $w>0$ of the weight in $0$ and  
for $\pp$-almost every $\eta$, the following sum is infinite:
    $$
    \sum_{x\in\eta}\pe(0\leftrightarrow x\,|\,W_0=w)\,
    	=\,\sum_{x\in\eta}\ew\Big[1-\e^{-wW\|x\|^{-\alpha}}\Big]\,,
    $$
    where the expectation in the r.h.s.~is taken with respect to $W$.
    Indeed, if we condition on the value of $W_0$, the presence of each edge $(0,x)$ becomes independent from the others. We can  therefore use the second Borel-Cantelli lemma to imply that there are infinitely many points in the configuration $\eta$ connected to 0.
By Campbell's theorem (cfr.~Theorem \ref{campbellthm}), 
the above sum is infinite
    for almost every $\eta$ if the following integral is:
    $$
    \int_{\R^d}\ew\Big[1-\e^{-wW\|y\|^{-\alpha}}\Big]\,{\rm d}y\,.
    $$
    Yet, in view of the inequality $1-\e^{-u}\geq(u\wedge 1)/2$ for $u \geq0$,
    it is enough to show the divergence of the integral
    \begin{equation}
        \label{integral}
        \int_{\R^d}\ew\big[{wW}{\|y\|^{-\alpha}}\wedge 1\big]\,{\rm d}y\,
        \geq\,
    \int_{\R^d}\ew\big[ \mathds{1}_{\{\|y\|^\alpha>wW\}}{wW}{\|y\|^{-\alpha}}  \big]\,{\rm d}y\,.
\end{equation}

Let us begin by showing that the integral above diverges when $\alpha\leq d$.
By applying Tonelli's theorem first and then passing to polar coordinates 
we can rewrite the right hand side of \eqref{integral} as
$$
    \ew\bigg[wW\int_{y:\|y\|>(wW)^{-\alpha}} \|y\|^{-\alpha}\,{\rm d}y \bigg]\,
    	=\,\ew\bigg[wW\sigma(S_{d-1})\int_{(wW)^{-\alpha}}^{\infty}r^{-\alpha+d-1}\,{\rm d}r \bigg]\,,
$$
where we denoted by $\sigma(S_{d-1})$ the surface area of the $(d-1)$--sphere of radius $1$. 
The last integral is divergent for $\alpha\leq d$.

Let us next assume that $\alpha>d$ and $\alpha(\tau-1)/d\leq 1$.
Since in this case $\tau\in(1,2)$, 
it follows that $\ew[W]=\infty$ and that there exists, thanks to \eqref{rocco}, some constant $c>0$ such that
$
\ew[W\1{W<s}]\,\geq\,c\,s^{2-\tau}\,
$
for all $s>0$.    
Using this inequality in the last integral of~\eqref{integral} we obtain
$$
\int_{\R^d}\ew\big[\mathds{1}_{\{W<\|y\|^\alpha/w\}}{wW}||y||^{-\alpha} \big]\,{\rm d}y\,
	\geq\,c\,w^{\tau-1}\int_{\R^d}\|y\|^{\alpha(1-\tau)}\,{\rm d}y\,.
$$
    Again, passing to polar coordinates, 
    we deduce that there exists some positive constant $c'$
    such that the integral in the right--hand side is bounded below by
    $c'\int_0^\infty r^{\alpha(1-\tau)+d-1}\,{\rm d}r\,,$
    which is infinite if $\alpha(\tau-1)/d\leq 1$.
\end{proof}

\begin{proof}[Proof of Theorem \ref{infinitedegree}]
	Thanks to Proposition \ref{obell} we know that the set of $\eta$'s for which
$\sum_{x\in\eta}\pe(0\leftrightarrow x\,|\,W_0=w)=\infty$ has $\pp$--measure $1$. Take an $\eta$ in this set and any point $y\in\eta$. We have 
$$
	\sum_{x\in\eta\setminus\lbrace y\rbrace}\pe(y\leftrightarrow x\,|\,W_y=w)
		=\sum_{x\in\eta\setminus\lbrace y\rbrace}\pe(0\leftrightarrow x\,|\,W_0=w)
		\frac{\ew\big[1-\e^{-wW||x-y||^{-\alpha}}\big]}{\ew\big[
			1-\e^{-wW||x||^{-\alpha}}\big]}
		\,=\,\infty\,,$$
	since  the fraction goes to 1 as $||x||$ goes to infinity. By Borel-Cantelli this shows that also $y$ has infinite degree $\pe$--almost surely. Since there is a countable number of points in $\eta$, we are done.
\end{proof}

\subsection{Polynomial degree}\label{secpolynomialdegree}

As in the previous section we will first prove a statement about the random variable $D_0$. The proof of Theorem \ref{degree} will be inferred as a consequence.
\begin{proposition}\label{propdegree}
	Suppose that $\a>d$ and $\gamma:=\a(\tau-1)/d>1$. Then, for $\pp$-almost every $\eta$, there exists a slowly varying function $\ell(\cdot)=\ell(\cdot,\eta)$, such that
	\begin{align}\label{notorious2}
	P^\eta(D_0>s)=s^{-\gamma}\ell(s)\,.
	\end{align}
\end{proposition}

The strategy for demonstrating  Proposition \ref{propdegree} follows the lines of the proof of Theorem 2.2 in \cite{DVH13}, which in turn follows \cite{Y06}. Analogously to Proposition 2.3 of  \cite{DVH13}, we analyse first the properties of the expectation of the degree of a point conditional to its weight $w>0$. Here it emerges one of the main problems when dealing with vertices that are randomly distributed in space: while in scale-free percolation on the lattice one can show by direct computations that the expectation of the degree is equal  to a constant times $w^{d/\alpha}$, in the continuum space case one has to finely control the fluctuations of the degree due to the irregularity of the point process. We show that these fluctuations are of order smaller than $w^{d/2\alpha}\log w$:

\begin{proposition}\label{marmotta}
	For $\pp$-almost every realization $\eta$ we have
	\begin{align}\label{dragotto}
	E^\eta[D_0\,|\,W_0=w]= \l c_0\,w^{d/\alpha}+O(w^{d/(2\a)}\log w)\,,
	\end{align}
	with $c_0:=v_d\Gamma(1-d/\alpha)\ew[W^{d/\alpha}]$. Here $v_d$ indicates the volume of the $d$-dimensional unitary ball and $\Gamma(\cdot)$ is the  gamma function. 
\end{proposition}

%
%
\begin{proof}[Proof of Proposition \ref{marmotta}]
We call $Z_w=Z_w(\eta):=E^\eta[D_0\,|\,W_0=w]$ and notice that $Z_w$ is a random variable that depends only on the Poisson process. Since $D_0=\sum_{x\in\eta}\1{0\leftrightarrow x}$, we can use
Campbell's  theorem (see \eqref{campbellexpectation} in Theorem \ref{campbellthm})
applied to the function $f(x):=\ew[1-\e^{-wW/\|x\|^\a}]$, 
to explicitly compute
\begin{align}\label{media}
\ep[Z_w]
	=\l\int_{\R^d}\ew\Big[1-\e^{-wW\|x\|^{-\a}}\Big]\,{\rm d}x
	=w^{d/\alpha}\ew\big[W^{d/\alpha}\big]\l\int_{\R^d}\big(1-\e^{-\|y\|^{-\a}}\big)\,{\rm d}y
	&=\l c_0 w^{d/\a}\,.
\end{align}
For the second inequality we have used Fubini and then made the change of variables $y=x(wW)^{-1/\alpha}$. The constant $c_0$ is the one appearing in the statement of the theorem; it is obtained by passing to polar coordinates and then using integration by parts.
We can use instead \eqref{campbellvariance}  in order to calculate the variance $ \vp$:
\begin{align}\label{varianzaa}
\vp(Z_w)
	=\l\int_{\R^d}\ew\Big[1-\e^{-wW\|x\|^{-\a}}\Big]^2\,{\rm d}x
	=\l c_1 w^{d/\a}\,.
\end{align}
To see the last equality it is sufficient to make the change of variables $y=xw^{-1/\alpha}$. We also notice that  $c_1\leq c_0$  because the expectation inside the integral is smaller than $1$, so the variance has to be smaller than the expectation.

We will show that for $\pp$-a.e.~$\eta$ we have
\begin{equation}\label{shanghai}
-1\leq\liminf_{w\to\infty}
\frac{Z_w-\ep[Z_w]}{\sqrt{\vp(Z_w)}\log w}
\leq\limsup_{w\to\infty}
\frac{Z_w-\ep[Z_w]}{\sqrt{\vp(Z_w)}\log w}
\leq 1\,,
\end{equation}
which implies \eqref{dragotto}.
We will only show the upper bound, 
the lower bound being very similar.
Let $\theta>0$. We use the exponential Chebyshev inequality to bound
\begin{align}\label{logiciel}
\pp(Z_w - \ep[Z_w]\geq \sqrt{\vp(Z_w)}\log w)
	&=\pp\big(\e^{\theta Z_w }\geq \e^{\theta ( \sqrt{\vp(Z_w)}\log w + \ep[Z_w])}\big)\nonumber\\
	&\leq  \exp\Big\{-\theta \big(\sqrt{\vp(Z_w)}\log w+\ep[Z_w]\big)+\log \ep\big[\e^{\theta Z_w}\big]\Big\}\,. 
\end{align}
Campbell's theorem in its exponential form (see \eqref{campbellexponential}) applied to the function 
$f(x):=\theta \ew[1-\e^{-wW\|x\|^{-\alpha}}]$ gives
\begin{align}\label{gianna}
\log \ep\big[\e^{\theta Z_w}\big]
	\,=\, \l\int_{\R^d} \Big(\exp\big\{\theta \ew[1-\e^{-wW\|x\|^{-\a}}]\big\}-1\Big) {\rm d}x\,.
\end{align}
If we restrict to values of $\theta$ smaller than $1$, a third-order Taylor expansion shows that
\begin{align*}
\exp\big\{\theta \ew\big[&1-\e^{-wW\|x\|^{-\a}}\big]\big\}-1\\
	&\leq \theta \ew\big[1-\e^{-wW\|x\|^{-\a}}\big] +\frac{ \theta^2}{2} \ew\big[1-\e^{-wW\|x\|^{-\a}}\big]^2 + O\big(\theta^3 \ew\big[1-\e^{-wW\|x\|^{-\a}}\big]\big)\,.
\end{align*}
Integrating over $x$ in the above inequality and in view of the expressions  
for the expectation \eqref{media} and the variance
\eqref{varianzaa}  of $Z_w$, we obtain
\begin{align*}
\log \ep\big[\e^{\theta Z_w}\big]
	\leq \theta \ep[Z_w]+\frac{ \theta^2}{2}\vp(Z_w)+O(\theta^3\ep[Z_w])\,.
\end{align*}
Inserting this estimate back into \eqref{logiciel} yields, for all $\theta\in[0,1]$,
\begin{align*}
\pp\big(Z_w - \ep[Z_w]\geq \sqrt{\vp(Z_w)}&\log w\big)
	\leq \exp\big\{-\theta \sqrt{\vp(Z_w)}\log w+\frac{ \theta^2}{2}\vp(Z_w)+O(\theta^3 \ep[Z_w])\big\}\,.
\end{align*} 
In particular, by choosing $\theta=\log w/\sqrt{\vp(Z_w)}$, we obtain
\begin{align}\label{rotari}
\pp\big(Z_w - \ep[Z_w]\geq \sqrt{\vp(Z_w)}&\log w\big)
	\leq \e^{-\log^2 w/4}\,.
\end{align}
If we consider the sequence $w_n=n$, 
we see that the quantity in the r.h.s.~of \eqref{rotari} is summable in $n$ and Borel-Cantelli tells that
\begin{equation}
	\limsup_{n\to\infty}
	\frac{Z_{n}-\ep[Z_{n}]}{\sqrt{\vp(Z_{n})}\log n}
	\,\leq\,1\,.
\end{equation}
The random variable $Z_w$ is increasing in $w$,
so that $Z_{\lfloor w\rfloor} \leq Z_w\leq Z_{\lfloor w\rfloor+1}$.
Moreover, in view of~\eqref{media} and~\eqref{varianzaa}, we 
see that the sequences $\ep(Z_n), \sqrt{\vp(Z_n)}$ and $\log n$
all satisfy 
$\lim_{n\to\infty}\frac{a_{n+1}}{a_n}\,=\,1\,.$
We can thus conclude that
$$\limsup_{n\to \infty}
\frac{Z_w-\ep(Z_w)}{\sqrt{\vp(Z_w)}\log w}
	\leq \limsup_{n\to\infty}\frac{Z_{\lfloor w\rfloor +1}-\ep(Z_{\lfloor w\rfloor})}
{\sqrt{\vp(Z_{\lfloor w \rfloor})}\log \lfloor w \rfloor}
	=\limsup_{n\to\infty}\frac{Z_{\lfloor w\rfloor +1}-\ep(Z_{\lfloor w\rfloor+1})}
{\sqrt{\vp(Z_{\lfloor w \rfloor+1})}\log (\lfloor w \rfloor+1)}
	\leq 1\,.
$$
%
%
\end{proof}

\begin{proof}[Proof of Proposition \ref{propdegree}]
We call $Y_w$ the random variable describing the value of $D_0$ conditioned on the event $\{W_0=w\}$: 
$$
Y_w\sim (D_0\,|\,W_0=w)\,.
$$ 
We split the integral 
\begin{align}\label{mare}
P^\eta(D_0>s)
	&=\int_0^{m(s)}\pe(Y_w>s)\,{\rm d}F(w)+\int_{m(s)}^\infty \pe(Y_w>s)\,{\rm d}F(w)=:A_1(s)+A_2(s)\,,
\end{align}
where 
\begin{align}\label{emme}
m(s):=\Big(\frac{s-\sqrt s\log^{ 2} s}{\l c_0}\Big)^{\a/d}\,,
\end{align}
$c_0$ being the same constant appearing in \eqref{dragotto}. We point out that this definition is slightly different from the equivalent appearing in \cite[Eq.~(2.11)]{DVH13} in order to control the fluctuations of the expected degree of $0$.
Since $P^\eta(D_0>s)$ is a monotone function, \eqref{notorious2} follows if we can prove that
\begin{align}\label{geneviev}
\lim_{t\to\infty}\frac{P^\eta(D_0>st)}{P^\eta(D_0>t)}=s^{-\gamma}
\end{align}
on a dense set of points (see \cite[Section VIII.8.]{F71}). We claim that $A_1(s)=A_1(\eta,s)$ does not contribute
to the regular variation of $P^\eta(D_0>s)$ for $\pp$-a.a.~$\eta$, that is
\begin{align}\label{potpot}
\lim_{s\to\infty} s^aA_1(s)=0\qquad \forall a>0,\,\pp-a.s. 
\end{align}
We will show how to get \eqref{potpot} at the end of the proof. Thanks to  \eqref{potpot}, we see that
in order to prove \eqref{geneviev} it is enough to verify that, for $s\in(0,\infty)$,
\begin{align}\label{geneviev2}
\lim_{t\to\infty}\frac{A_2(st)}{A_2(t)}=s^{-\gamma}\,.
\end{align}
On the one hand we clearly have that, for all $t>0$,
\begin{align*}
A_2(t)\leq 1-F(m(t))\,.
\end{align*}
On the other hand, for all $\varepsilon>0$, it holds
\begin{align*}
A_2(t)
	&\geq 1-F\big((1+\varepsilon)m(t)\big)-\int_{(1+\varepsilon)m(t)}^\infty P^\eta(Y_w\leq t)\,{\rm d}F(w)\nonumber\\
	&\geq 1-F\big((1+\varepsilon)m(t)\big)-P^\eta(Y_{(1+\varepsilon)m(t)}\leq t)\,.
\end{align*}
By \eqref{dragotto}, $E^\eta[Y_{(1+\varepsilon)m(t)}]>t(1+\varepsilon d/2\alpha)$ when  $t$ is sufficiently large (depending on $\eta$); this allows us to use Chebychev inequality and bound
\begin{align*}
P^\eta(Y_{(1+\varepsilon)m(t)}\leq t)
	\leq\frac{V^\eta(Y_{(1+\varepsilon)m(t)})}{(E^\eta[Y_{(1+\varepsilon)m(t)}]-t)^2}
	\leq\frac{E^\eta[Y_{(1+\varepsilon)m(t)}]}{(E^\eta[Y_{(1+\varepsilon)m(t)}]-t)^2}
	\leq \frac{C}{\varepsilon^2 t}
	=o(1)\,,
\end{align*}
where $V^\eta$ represents the variance w.r.t.~$\pe$ and  the second inequality follows form the fact that $Y_w$ is a sum of independent indicator functions.
Putting all together and using that $\lim_{t\to\infty}m(st)/m(t)=s^{\alpha/d}$, we obtain
\begin{align*}
\lim_{t\to\infty}\frac{A_2(st)}{A_2(t)}
	=\lim_{t\to\infty}\frac{1-F(m(st))}{1-F(m(t))}
	=s^{-\alpha(\tau-1)/d}
\end{align*}
as we wanted to prove.

We are just left to show \eqref{potpot}.
We start by upper-bounding $A_1(s)$ by
\begin{align}\label{neve}
P^\eta(Y_{m(s)}>s)=P^\eta\Big(\sum_{y\in\eta}\1{y\leftrightarrow 0}-E^\eta[\1{y\leftrightarrow 0}\,|\,W_0=m(s)]>s-E^\eta[Y_{m(s)}]\,\Big|\,W_0=m(s)\Big).
\end{align}
By taking $s$ sufficiently large and thanks to Proposition \ref{marmotta}, we can make $s-E^\eta[Y_{m(s)}]$ positive, allowing us to apply Bernstein's inequality. 
We obtain
\begin{align*}
A_1(s)	
	&\leq \exp\Big\{-\frac 12
		\frac{(s-E^\eta[Y_{m(s)}])^2}
		{E^\eta[Y_{m(s)}]+(s-E^\eta[Y_{m(s)}])/3}\Big\}\,,
\end{align*}
where for the first term in the denominator of the exponent we have used the inequality
\begin{align*}
E^\eta\Big[\Big(\1{y\leftrightarrow 0}-E^\eta[\1{y\leftrightarrow 0}\,|\,W_0=m(s)]\Big)^2\,\Big|\,W_0=m(s)\Big]
	\leq E^\eta\big[\1{y\leftrightarrow 0}\,\big|\,W_0=m(s)\big]\,.
\end{align*}
Since $E^\eta[Y_{m(s)}]=s-\sqrt s \log^2s+O(\sqrt s \log s)$, we get $A_1(s)\leq \e^{-\log^4s/4}$
and \eqref{potpot} is proven.
\end{proof}


\begin{proof}[Proof of Theorem \ref{degree}]
	Take any $x\in\eta$. We first prove a lower bound on $\pe(D_x>s)$. First of all we claim that
		\begin{align*}
		\pe(D_x>s\,|\,W_x=w)
			\geq  \pe\big(D_0>s+1\,|\,W_0= c(x,\eta)\cdot w\big)\,,
		\end{align*}
		where $c(x,\eta):=(1+\|x\|/\|y\|)^{-\alpha}$ 
		and $y\in\eta$ is the closest point to the origin in $\eta$. 
In order to prove the claim we write $D_x=\sum_{z\in\eta,\,z\neq x}\1{x\leftrightarrow z}$ and $D_0\leq \sum_{z\in\eta,\,z\neq x}\1{0\leftrightarrow z}+1\,$. Since the indicator functions in the first sum are mutually independent under $\pe(\cdot\,|\,W_x=w)$ and those in the second sum are mutually independent under $\pe(\cdot\,|\,W_0=c(x,\eta)\cdot w)$, it will be sufficient to show that, for all $z\in\eta\setminus \{x\}$,
\begin{align*}
\pe(\1{x\leftrightarrow z}=1\,|\,W_x=w)
\geq \pe\big(\1{0\leftrightarrow z}=1\,|\,W_0=c(x,\eta)\cdot w\big)\,,
\end{align*}
since this ensures that we can construct a coupling such that $D_x\geq D_0-1$.
We calculate
\begin{align*}
\pe(\1{x\leftrightarrow z}=1\,|\,W_x=w)
&=\ew\big[1-\e^{-wW_z\|z-x\|^{-\alpha} }\big]\\
&\geq \ew\big[1-\e^{-c(x,\eta)\cdot w W_z\|z\|^{-\alpha} }\big]
=\pe(\1{0\leftrightarrow z}=1\,|\,W_0=c(x,\eta)\cdot w)\,,
\end{align*}
where for the inequality we have used the fact that $\|z-x\|/\|z\|\leq (\|y\|+\|x\|)/\|y\|$.
	
	Thanks to the claim we can now bound
	\begin{align*}
	\pe(D_x>s)
		=\int_0^\infty \pe(D_x>s\,|\,W_x=w)\,{\rm d}F(w)
		\geq \int_0^\infty \pe\big(D_0>s\,|\,W_0=w\cdot c(x,\eta)\big)\,{\rm d}F(w)\,.
	\end{align*}
	If the weights follow simply a Pareto distribution with exponent $\tau$, then with a change of variables $u=c(x,\eta)\cdot w$ we would obtain $\pe(D_x>s)\geq c(x,\eta)^{\tau-1} \pe\big(D_0>s+1)$ and we would be done thanks to Proposition \ref{propdegree}. In the general case we have to proceed as in the proof of  \ref{propdegree} by splitting the integral in the r.h.s.~of the last display into the integral between $0$ and $m(s)/c(x,\eta)$ plus the integral between $m(s)/c(x,\eta)$ and $+\infty$, where $m(s)$ is defined in \eqref{emme}. We call the first part $\tilde A_1(s)$ and the second part $\tilde A_2(s)$ similarly to \eqref{mare}.
	Following step by step what we did in the proof of Proposition \ref{propdegree} (see  \eqref{neve} and below), we notice that $\tilde A_1(s)\leq P^\eta(D_0>s+1\,|\,W_0=m(s))\leq \e^{-(\log s)^4/4}$, which does not contribute to the regular variation of the sum.
	Finally (see \eqref{geneviev2} and the argument below) we have that $\lim_{t\to\infty}\tilde A_2(st)/\tilde A_2(t)=s^{-\gamma}$, since $\tilde A_2(t)$ is upper bounded by $1-F(m(t)/c(x,\eta))$  and lower bounded by $1-F((1+\varepsilon)m(t)/c(x,\eta))-P^\eta(D_0\leq t+1\,|\,W_0=(1+\varepsilon)m(t))$, the last summand being an $o(t)$.
	We conclude therefore that there exists a slowly varying function $\ell_1(\cdot)=\ell_1(\cdot,x,\eta)$ such that
	\begin{align*}
	\pe(D_x>s)\geq \ell_1(s)s^{-\gamma}\,.
	\end{align*}
	An upper bound $\pe(D_x>s)\leq \ell_2(s)s^{-\gamma}$ for some other slowly varying function $\ell_2=\ell_2(x,\eta)$ can be obtained in a completely specular way. These two bounds yield the desired result.
\end{proof}

\section{Clustering coefficient}\label{secclusteringcoefficient}


\begin{proof}[Proof of Theorem \ref{thmcc}]
The idea consists in approximating $\CC_n$ with the sum of independent random variables in order to use the standard law of large numbers. 

First of all we divide $\R^d$ into disjoint mesoscopic boxes of side-length $m>0$, one of which is centered in the origin (the superposition of the sides of the boxes is of no importance). 
For a point $x\in\R^d$ we call $Q_m(x)$ the unique $m$-box containing $x$.
We also fix $\delta>0$ small and divide each of the boxes $Q_m$ as $Q_m=\overline Q_m\cup \partial Q_m$, where $\overline Q_m=\overline Q_m(\delta):=\{x\in Q_m:\,\|x-y\|>\delta m,\;\,\forall y\in Q_m^c\}$ are the interior points of the box and $\partial Q_m=\partial Q_m(\delta):=Q_m\setminus \overline Q_m$  is the $\delta$-frame of the box.
For a realisation of our graph $G=(V,E)$ and a point $x\in V$, we define
\begin{align*}
\CCc^m(x)=\CCc^{m,\delta}(x):=
\begin{cases}
0\qquad&\mbox{if }x\in\partial Q_m(x)\\
0\qquad&\mbox{if }x\leftrightarrow Q_m^c(x)\\
\CC(x)&\mbox{otherwise.}
\end{cases}
\end{align*}
For $n>m$ we finally define the $(m,\delta)$-truncated clustering coefficient as
\begin{align*}
\CCc^m_n=\CCc^{m,\delta}_n:=\frac{1}{\l n^d}\sum_{x\in V_n}\CCc^m(x)\,.
\end{align*}

The idea is now to approximate $\CC_n$ by $\CCc_n^m$, to show that   $\CCc_n^m$ converges thanks to the law of large numbers to $\E[\CCc_m^m]$ and that this value is close to the desired $\E_0[\CC(0)]$. This is formalized in the next three statements, which are valid under the hypothesis of Theorem \ref{thmcc}:

\begin{proposition}\label{grandegiunone}
	$\P$-almost surely we have 
	\begin{align*}
	\limsup_{n\to\infty}|\CC_n-\CCc_n^m|\leq c_1\delta + c_2 (\delta m)^{d-\alpha}\,,
	\end{align*}
	for some constants $c_1,\,c_2>0$ that do not depend neither on $m$ nor on $\delta$.
\end{proposition}
\begin{lemma}\label{grandegiove}
	$\P$-almost surely we have
	\begin{align*}
	\lim_{n\to\infty}\CCc^m_n = \E[\CCc^m_m]\,.
	\end{align*}
\end{lemma}

\begin{lemma}\label{grandeapollo}
	There exist constants $c_1,c_2>0$ not depending on $m$ nor on $\delta$ such that 
	\begin{align*}
	\Big|\E\big[\CCc^m_m\big]-\E_0[\CC(0)]\Big|	
		\leq c_1\delta + c_2 (\delta m)^{d-\alpha}\,.
	\end{align*}
\end{lemma}
The proofs of Lemma \ref{grandegiove} and of Lemma \ref{grandeapollo} are pretty straight-forward and are collected in Section \ref{bolsonero}. The proof of Proposition \ref{grandegiunone} is much more involved and is the object of Section \ref{provona}. The convergence of $\CC_n$ to $\E_0[\CC(0)]$ 
is now concluded putting together the results of Propositions  \ref{grandegiunone}, Lemma \ref{grandegiove} and Lemma \ref{grandeapollo}, and then letting first $m\to\infty$ and then $\delta\to0$. In order to conclude the proof of Theorem \ref{thmcc} we are only left to prove that $\E_0[\CC(0)]>0\,$:

\begin{lemma}\label{localcc}
	The local clustering coefficient of the origin under $\P_0$ has positive expectation:
	$$
	\E_0\big[\CC(0)\big]
	\,>\,0\,.
	$$
\end{lemma}
\noindent Also the proof of this last lemma is postponed to Section \ref{bolsonero}.
\end{proof}

\subsection{Proof of the lemmas}\label{bolsonero}
\begin{proof}[Proof of Lemma \ref{grandegiove}]
	We write $n=mq+r$, with $q\in \N$ and $0\leq r<m$. We call $Q(1),\dots,Q(q^d)$ the $q^d$ boxes of side-length $m$ fully contained in $B_n$ and $H=\big(\cup_{j}Q(j)\big)^c$. We also define
	$X(j):=(1/\nu m^d)\sum_{x\in Q(j)\cap \eta}\CCc^m(x)$ for $j=1,\dots,q^d$. Then
	\begin{align*}
	\CCc^m_n=\Big(\frac{mq}{n}\Big)^d\frac{1}{q^d}\sum_{j=1}^{q^d}X(j)+\frac{1}{\nu n^d}\sum_{x\in H\cap V_n}\CCc^m(x)\,.
	\end{align*}
	Since the $X(j)$'s are i.i.d.~random variables with finite $\P$-expectation, as $n$ goes to infinity, the first summand in the r.h.s.~converges $\P$-a.s.~to $\E[X(1)]=\E[\CCc^m_m]$ by translation invariance. The second summand converges almost surely to $0$ since $\CCc^m(x)\leq 1$ and the number of points in $H\cap V_n$ follows a Poisson distribution with parameter $n^d-(mq)^d=O(n^{d-1})$.
\end{proof}
\begin{proof}[Proof of Lemma \ref{grandeapollo}]
	By the Slivnyak-Mecke Theorem (see Theorem \ref{mecke}) with $n=1$ we have 
	\begin{align*}
	\E\Big[\frac{1}{\l m^d}\sum_{x\in V_m}\CC(x)\Big]
	=\frac {1}{\l m^d} \int_{B_m}\E_x[\CC(x)]\l{\rm d}x
	=\E_0[\CC(0)]\,,
	\end{align*}
	where we have used translation invariance for the last equality. 
	On the other hand,
	\begin{align*}
	\Big|\E\big[\CCc^m_m\big]-  \E\Big[\frac{1}{\l m^d}\sum_{x\in V_m}\CC(x)\Big] \Big|	
	= \frac{\E[W_m+U_m]}{\l m^d} 
	\end{align*}
	where $W_m:=\#\{x\in\partial B_m(x)\cap \eta\}$ and $U_m:=\#\{x\in\overline B_m(x)\cap\eta:\,x\leftrightarrow (B_m(x))^c\}$.
	Since $W_m$ follows a Poisson law of parameter $(1-(1-2\delta)^d)\l m^d$, $\E[W_m]/(\l m^d)$ is smaller than $2d\delta$. The fact that $\E[U_m]/(\l m^d)$ is bounded from above by $c (\delta m)^{d-\alpha}$ will be proved in Proposition \ref{ututtecose}. Putting all together we obtain
	\begin{align*}
	\Big|\E\big[\CCc^m_m-\E_0[\CC(0)]
	\Big]\Big|
	\leq c_1\delta +c_2 (\delta m)^{d-\alpha}\,.
	\end{align*}
\end{proof}
\begin{proof}[Proof of Lemma \ref{localcc}]
	Write $\mathcal B$ for $\mathcal B_1(0)$. Let us consider the following events:
	\begin{align*}
	\mathcal{E}_1\,&=\,\big\lbrace\,
	|\eta\cap \mathcal B|=2
	\,\big\rbrace\,,\\
	\mathcal{E}_2\,&=\,\big\lbrace\,
	\text{the points in } \eta\cup\{0\} \cap \mathcal B \text{ form a clique}
	\,\big\rbrace\,,\\
	\mathcal{E}_3\,&=\,\big\lbrace\,
	0 \text{ has no neighbors outside } \mathcal B
	\,\big\rbrace\,.
	\end{align*}
	Under the event $\mathcal{E}_1\cap\mathcal{E}_2\cap\mathcal{E}_3$,
	the local clustering coefficient of the origin is 1. Therefore,
	$$\E_0\big[
	\CC(0)
	\big]\,\geq\,\P_0\big(
	\CC(0)=1
	\big)\,\geq\,\P_0\big(
	\cE_1\cap\cE_2\cap\cE_3
	\big)\,.$$
	Given the value of $W_0$, the events $\cE_1\cap\cE_2$ and $\cE_3$ are independent.
	Hence,
	$$
	\P_0\big(\cE_1\cap\cE_2\cap\cE_3\big)\,
	=\,\int_0^\infty\P_0\big(\cE_1\cap\cE_2\,\big|\,W_0=w\big)\P_0\big(\cE_3\,\big|\,W_0=w\big)\mathrm{d} F(w)\,.
	$$
	In order to bound from below the second of the probabilities in the integral, we notice that $\1{0\not\leftrightarrow \mathcal B^c}=\prod_{x\in \eta\cap \mathcal B^c}\1{0\not\leftrightarrow x}$ and
	we apply subsequently Jensen's inequality and Campbell's theorem,
	thus obtaining:
	\begin{align*}
	\P_0\big(\cE_3\,\big|\,W_0=w\big)
	&=\ep_0\Big[ \prod_{x\in \eta\cap \mathcal B^c}
	\ew\Big[\e^{-\frac{w W_x}{\|x\|^\alpha}}\Big] \Big]
	\geq \exp\Big\{-\ep_0\Big[\sum_{x\in \eta\cap \mathcal B^c}
	\frac{w \ew[W_y]}{\|y\|^\alpha}\Big]\Big\}
	=\e^{-c\ew[W]w\l } 
	\end{align*}
	for some $c>0$ that does not depend on $w$. We can simplify things by just considering values $w>1$ and notice that clearly $\P_0(\cE_1\cap\cE_2\,|\,W_0=w)$ is uniformly bounded from below by a constant. The result of the lemma follows.
\end{proof}

\subsection{Proof of Proposition \ref{grandegiunone}}\label{provona}

We begin with an elementary lemma.
\begin{lemma}\label{equitron} 
	For $\pp$-a.a.~$\eta$ there exists $ \bar n\in\N$ such that, for all $n\geq \bar n$,
	\begin{align*}
	N_n(\eta)\in\big[\l n^d-\sqrt{\l n^d}\log n,\,\l n^d+\sqrt{\l n^d}\log n\big]\,.
	\end{align*}
\end{lemma}
\begin{proof}
	For a Poisson random variable $X$ of parameter $\mu>0$ it holds the bound $P(|X-\mu|>\varepsilon)<\exp\{-\varepsilon^2/(2\mu)\}$ for all $0<\varepsilon<\mu$. Since under $\pp$ the number of points falling in $B_n$ follows a Poisson distribution of parameter $\l n^d$, the conclusion follows by a simple Borel-Cantelli argument (in fact, it is possible to  further improve the statement).
\end{proof}
For simplicity we will consider a sequence of $n$'s of the form $n=qm$, where $q\in\N$. It is in fact easy to get convinced that for all other $n$ of the form $n=qm+r$ with $r\in[0,m)$, what happens in the area $B_n\setminus B_{qm}$ is negligible in the limit $n\to\infty$.
We bound 
\begin{align}\label{circe}
\big|\CC_n-\CCc_n^m\big|
	&=\Big|\frac{1}{N_n}-\frac{1}{\l n^d}\Big|\sum_{x\in V_n}\CC(x)
		+\frac{1}{\l n^d}\Big(\sum_{x\in V_n}\CC(x)-\CCc^m(x)\Big)\nonumber\\
	&\leq \Big|1-\frac{N_n}{\l n^d}\Big|+\frac{1}{\l n^d}\big(W_n+U_n\big)\,,
\end{align}
where $W_n=W_n(m):=\#\{x\in V_n:\,x\in\partial Q_m(x)\}$ and 
$$
	U_n=U_n(m)
		:=\#\{x\in V_n:\,x\in\overline Q_m(x),\,x\leftrightarrow (Q_m(x))^c\}\,.
$$

The first summand in \eqref{circe} tends $\P$-a.s.~to $0$ as $n\to \infty$ by Lemma \ref{equitron}.
We notice then that under $\P$ the random variable $W_n$ has a Poisson distribution of parameter $(1-(1-2\delta)^d)\l n^d$, which can be dominated by a Poisson of parameter $ 2d\delta\l n^d$. Reasoning as in the the proof of Lemma \ref{equitron} it is possible to show that for all $n$ big enough $W_n$ is smaller than, for example, $4d\delta \l n^d$, so that $\limsup_{n\to\infty}n^{-d}W_n\leq c_1{\delta}$.

It remains to show that $\P$-a.s.
\begin{align}\label{ulimit}
\limsup_{n\to\infty}\frac{1}{\l n^d}U_n\leq c_2 (\delta m)^{d-\alpha}\,.
\end{align}
We study expectation and variance of the variable $U_n$.
\begin{proposition}\label{ututtecose}
	There exists a constant $c>0$ such that
	\begin{align}
	\E[U_n]&\leq c\,n^d(\delta m)^{d-\alpha}\label{umean}\\
	\mathbb V(U_n)&\leq c\, n^{3d-\alpha}\,.\label{uvariance}
	\end{align}
	where $\mathbb V$ denotes the variance w.r.t.~$\P$.
\end{proposition}
Before proving the proposition we conclude the argument.
By Chebychev inequality
\begin{align*}
\P(U_n>2c\,n^d(\delta m)^{d-\alpha})
\leq \frac{\mathbb V(U_n)}{c^2n^{2d}(\delta m)^{2(d-\alpha)}}
\leq \tilde c\,(\delta m)^{2(\alpha-d)} n^{d-\alpha}\,,
\end{align*}
for some $\tilde c>0$.
We would like to end the proof via Borel-Cantelli, but the r.h.s.~of the last display is not summable in $n$. Nevertheless $U_n$ is an increasing sequence, therefore we can proceed as follows. Recall that $n=qm$. By taking the sequence $n_k:=2^km$ we have
\begin{align*}
\P(U_{n_k}>2c\,{n_k}^d(\delta m)^{d-\alpha})
\leq \tilde c\,(\delta m)^{2(\alpha-d)} (m2^k)^{d-\alpha}\,.
\end{align*}
which is summable in $k$. Hence there exists almost surely a $\bar k$ such that $U_{n_k}\leq 2c\,{n_k}^d(\delta m)^{d-\alpha}$ for all $k\geq \bar k$. Now take any $n>2^{\bar k}m$ and call $K$ the integer such that $2^Km\leq n<2^{K+1}m$. By monotonicity we get
\begin{align*}
U_n
\leq U_{2^{K+1}m}
\leq 2 c (2^{K+1}m)^d(\delta m)^{d-\alpha}
\leq 2^{d+1}c\,(\delta m)^{d-\alpha}n^d
\end{align*}
and \eqref{ulimit} follows.

\begin{proof}[Proof of Proposition \ref{ututtecose}]
	Call $Q(1),\dots,Q(q^d)$ the boxes of side length $m$ contained in $B_n$. For the expectation we have
	\begin{align}\label{ranescalc}
	\E[U_n]
		&=\E\Big[\sum_{j=1}^{q^d}\sum_{x\in \overline Q(j)\cap\eta}\1{x\leftrightarrow(Q_m(x))^c}\Big]\nonumber\\
		&= q^d\E\Big[ \E\Big[ \sum_{x\in V_{m(1-2\delta)}} \1{x\leftrightarrow B_m^c} \,\Big|\,V_m \Big] \Big]
		\leq Q^d\E\big[ N_{m(1-2\delta)}\P_0(0\leftrightarrow B_{\delta m}^c) \big] \,.
	\end{align}
	In the second equality we have used the translation invariance of the model  and then the tower property of expectation by conditioning on the position of the points of the  Poisson process inside $B_m$. For the inequality we have upper bounded the probability that a point $x\in V_{m(1-2\delta)}$ is connected to the exterior of $B_m$ by the probability that, under the Palm measure, the origin  is connected with some point in $B_{\delta m}^c$ (since for each $x\in V_{m(1-2\delta)}$ all the points in $B_m^c$ are outside a box of side $\delta m$ centered in $x$). We estimate this probability as follows:
	\begin{align}\label{palmotta}
	\P_0(0\leftrightarrow B_{\delta m}^c)
		&=1-\E_0\Big[ \prod_{y\in V\setminus V_{\delta m}}\ew\Big[\e^{-\frac{W_0W_y}{\|y\|^\alpha}}\Big|\,W_0\Big] \Big]\nonumber\\
		&\leq 1-\exp\Big\{-\E_{0}\Big[\sum_{y\in  V\setminus V_{\delta m}}\frac{W_0\ew[W_y]}{\|y\|^\alpha}\Big]\Big\}	
		\leq \int_{y\in B_{\delta m}^c}\frac{\ew[W]^2}{\|y\|^{\alpha}}\,{\rm d} y
		= c'\,(\delta m)^{d-\alpha}
	\end{align}
	for some $c'>0$ that does not depend neither on $m$ nor on $\delta$.
	In the first line we have used the fact that, conditioned on the weight of point $0$, the presence of connections between $0$ and the other points of the Poisson point process become independent. For the second line we have applied twice Jensen's inequality. For the last line we have first used the inequality $1-\e^{-x}\leq x$, for $x>0$, and then we have applied Campbell's theorem (see Theorem \ref{campbell}) and finally we have used polar coordinates in order to evaluate the integral.
	Plugging this back into \eqref{ranescalc} we obtain \eqref{umean}.

	We move to the variance of $U_n$.	For $x\in \eta$ set
	\begin{align*}
	A(x):=\1{x\in \overline Q_m(x)\,,\;x\leftrightarrow(Q_m(x))^c}\,
	\end{align*}
	so that $U_n=\sum_{x\in V_n}A(x)$.
	Notice that by Slivnyak-Mecke theorem (see Theorem \ref{mecke}) one has  $\E[U_n]=\l\int_{B_n}\E_{x}[A(x)]{\rm d}x$. 
	On the other hand we can write
	\begin{align}\label{varianza}
	\mathbb V(U_n)
		=\E\big[U_n^2\big]-\E\big[U_n\big]^2
		=\E\big[U_n\big]+\E\Big[\sum_{x\neq y\in V_n}A(x)A(y)\Big]-\E\big[U_n\big]^2\,.
	\end{align}
	Applying now Theorem \ref{mecke} to the function $f(x,y,\tilde\eta):=\1{x,y\in B_n}E^{\tilde\eta\cup \{x\}\cup\{y\}}[A(x)A(y)]$, we can evaluate the second summand in the r.h.s.:
	\begin{align}\label{pidgeon}
	\E\Big[\sum_{x\neq y\in V_n}A(x)A(y)\Big]
		&=\l^2\int_{B_n}\int_{B_n}\E_{x,y}\big[A(x)A(y)\big]\,{\rm d}x{\rm d}y\notag\\
		&=\l^2\int_{B_n}\int_{B_n}C(x,y)\,{\rm d}x{\rm d}y+{\E[U_n]}^2+\l^2R\,,
	\end{align}
	where $C(x,y)=\E_{x,y}[A(x)A(y)]-\E_{x,y}[A(x)]\E_{x,y}[A(y)]$ and $R$ is the rest given by 
	$$
	R
		=\int_{B_n}\int_{B_n}\E_{x,y}[A(x)]\E_{x,y}[A(y)]-\E_{x}[A(x)]\E_{y}[A(y)]\,{\rm d}x\,{\rm d}y\,.
	$$
	We notice that for $y\in Q_m(x)$ one has $\E_{x,y}[A(x)]=\E_{x}[A(x)]$ and for $y\not\in Q_m(x)$ it holds 
	$\E_{x,y}[A(x)]\leq \E_{x}[A(x)]+\ew[1-\exp\{W_xW_y\|x-y\|^{-\alpha}\}]\leq \E_{x}[A(x)]+c_1\|x-y\|^{-\alpha}$ for some $c_1>0$. This, together with the bound $\int_{B_n\setminus Q_m(x)}\|x-y\|^{-\alpha}\,{\rm d}y\leq c_2 (\delta m)^{d-\alpha}$ for some $c_2>0$, yields
		\begin{align*}
		R
			\leq c_3 (\E[U_n](\delta m)^{d-\alpha}+n^d(\delta m)^{-2\alpha+d})
			\leq c_4 n^d(\delta m)^{2(d-\alpha)}\,.
		\end{align*}
	 We are left to bound the correlation  $C(x,y)$. We introduce the event
	\begin{align*}
	\cE(x,y):=\{\mbox{Neither $x$ nor $y$ have neighbors at distance larger than $\|x-y\|/2$}\}.
	\end{align*}
	The random variables $A(x)$ and $A(y)$ are independent under the event $\cE(x,y)$. Therefore
	\begin{align}\label{covarianza}
	C(x,y)
		&\leq \E_{x,y}\big[A(x)|\cE(x,y)\big] \E_{x,y}\big[A(y)|\cE(x,y)\big]\P(\cE(x,y))\notag\\
		&\qquad\qquad\qquad+\P_{x,y}\big(\cE(x,y)^c\big)-\E_{x,y}[A(x)]\E_{x,y}[A(y)]\nonumber\\
		&\leq \E_{x,y}[A(x)]\E_{x,y}[A(y)]\Big(\frac{1}{\P_{x,y}(\cE(x,y))}-1\Big)+\P_{x,y}\big(\cE(x,y)^c\big)\nonumber\\
		&\leq \P_{x,y}\big(\cE(x,y)^c\big)\Big(\frac{1}{\P_{x,y}(\cE(x,y))}+1\Big)\,,
	\end{align}
	where in the second line we have used the inequality $\E_{x,y}[W|Z]\leq \E_{x,y}[W]/\P_{x,y}(Z)$ and in the last line we have just bounded $\E_{x,y}[A(x)]$ and $\E_{x,y}[A(y)]$ by $1$. We also observe that
	\begin{align}\label{pushotti}
	\P_{x,y}\big(\cE(x,y)^c\big)
		&\leq \P_{x,y}(x\leftrightarrow \mathcal B_{\|x-y\|/2}(x))+\P_{x,y}(y\leftrightarrow \mathcal B_{\|x-y\|/2}(y))\nonumber\\
		&\leq \P_{x,y}(x\leftrightarrow y)+2\P_{0}(0\leftrightarrow \mathcal B_{\|x-y\|/2}(0))\notag\\
		&\leq c'\,\|x-y\|^{d-\alpha}
	\end{align}
	for some $c'>0$, where the last inequality can be obtained as in \eqref{palmotta}. From \eqref{covarianza} and \eqref{pushotti} it follows that whenever we consider $x,y\in\R^d$ such that, for example, $c'\,\|x-y\|^{d-\alpha}\leq 1/2$, one gets $C(x,y)\leq c''\,\|x-y\|^{d-\alpha}$. Hence, calling $R:=(2c')^{1/(\alpha-d)}$, we can bound 
	\begin{align*}
	\int_{B_n}\int_{B_n}C(x,y)\,{\rm d}x{\rm d}y
		\leq \int_{B_n}\Big[c'' R^d+\int_{y\not\in \mathcal B_{R}(x)}c'\|x-y\|^{d-\alpha}\,{\rm d}y\Big]\,{\rm d}x
		\leq c\,n^{3d-\alpha}\,
	\end{align*}
	for some constant $c>0$.
	The proof is finished by putting together \eqref{varianza}, \eqref{umean}, \eqref{pidgeon} and this last estimate. 
\end{proof}

%

\bigskip

\appendix
\section{Campbell and Slivnyak-Mecke theorems}\label{campbell}
In this section we present versions of Campbell's theorem and of the so-called extended Slivnyak-Mecke theorem in the simple case of an homogeneous Poisson point process. For a version of Theorem \ref{campbellthm} similar to the one presented here see \cite[Section 3.2]{K92}. For a simple proof of Theorem \ref{mecke} see  \cite[Theorem 13.3]{MW03}. More complete and general versions of these theorems in the framework of Palm theory can be found in \cite[Chap.~13]{DVJ}.

\begin{theorem}[Campbell theorem, homogeneous case]\label{campbellthm}
Let $\eta$ be a homogeneous Poisson
point process on $\R^d$ with intensity $\l$ under measure $\pp$, with $\ep$ the associated expectation. Let $f:\R^d\to\R$ 
be a measurable function.    The random sum $S=\sum_{x\in\eta} f(x)$ is absolutely convergent with probability one, if and only if 
    $$
    \int_{\R^d} |f(y)|\wedge 1\,dy\,<\,\infty\,.
    $$
If the previous integral is finite, it holds
\begin{align}\label{campbellexponential}
\ep[\e^{\theta S}]\,
	&=\,\exp\bigg(\l\int_{\R^d}(\e^{\theta f(y)}-1)dy\bigg)
	\end{align}
for any complex $\theta$ such that the integral in the r.h.s.~converges.
Moreover
\begin{align}\label{campbellexpectation}
        \ep[S]\,&=\,\l\int_{\R^d}f(y)\,dy
\end{align}
and if the last integral is finite, then        
\begin{align}\label{campbellvariance}
        \vp(S)\,&=\,\l\int_{\R^d}f(y)^2\,dy\,,
\end{align}
where $\vp$ is the (possibly infinite) variance with respect to $\pp$.
\end{theorem}

\begin{theorem}[Extended Slivnyak-Mecke theorem]\label{mecke}
	 Let $\eta$ be a homogeneous Poisson
	point process on $\R^d$ with intensity $\l$ under measure $\pp$, with $\ep$ the associated expectation. For $n\in\N$ and for any positive function $f:(\R^d)^n\times (\R^d)^\N\to[0,\infty)$ it holds
	\begin{align*}
	\ep\Big[\sum_{x_1,\dots, x_n\in\eta}^{\neq} f(x_1,\dots,x_n;\,\eta\setminus\{x_1,\dots,x_n\})\Big]
		=\l^n\int_{\R^d}\dots\int_{\R^d} \ep[f(x_1,\dots,x_n;\,\eta)]\,{\rm d}x_1\dots{\rm d}x_n\,.
	\end{align*}
	where the $\neq$ sign over the sum means that $x_1,\dots,x_n$ are pairwise distinct.
\end{theorem}

\medskip


\section*{Acknowledgments}
The second author would like to thank Quentin Berger for useful discussions.

This work was carried out with the financial support of the French Research Agency (ANR), project ANR-16-CE32-0007-01 (CADENCE).

\medskip

\bibliography{bibliografia}
\bibliographystyle{alpha}

\end{document}